\documentclass[11pt]{article}
\usepackage{graphicx}
\usepackage{subfig}
\usepackage{newtxmath}
\usepackage{newtxtext,amsmath,amsthm,mathrsfs,amsfonts,mathtools,braket,feynmf,eucal,secdot,graphicx,labelfig}
\usepackage{tipa,overpic,microtype,float,booktabs,xcolor,cancel,slashed}
\usepackage[letterpaper]{geometry}
\usepackage[english]{babel}
\usepackage{caption}
\usepackage[hang,flushmargin]{footmisc}
\usepackage[pdftitle={Thesis},pdfauthor={Nhat-Minh Doan},ocgcolorlinks,linkcolor=blue,citecolor=orange,urlcolor=red]{hyperref}
\usepackage[tiny]{titlesec}\titlespacing{\section}{0pt}{12pt}{0pt}\titlespacing{\subsection}{0pt}{6pt}{0pt}
\usepackage{sectsty}
\usepackage{setspace}
\chaptertitlefont{\Huge}
\long\def\symbolfootnote[#1]#2{\begingroup
\def\thefootnote{\fnsymbol{footnote}}\footnote[#1]{#2}\endgroup}

\DeclareMathOperator{\arccosh}{arccosh}
\DeclareMathOperator{\arcsinh}{arcsinh}

\newcommand\smallO{
  \mathchoice
    {{\scriptstyle\mathcal{O}}}
    {{\scriptstyle\mathcal{O}}}
    {{\scriptscriptstyle\mathcal{O}}}
    {\scalebox{.7}{$\scriptscriptstyle\mathcal{O}$}}
  }
\makeatletter
\def\@cite#1#2{{\normalfont[{\bfseries#1\if@tempswa , #2\fi}]}}
\makeatother

\newcommand{\bi}{\begin{itemize}}
\newcommand{\ei}{\end{itemize}}
\newcommand{\bq}{\begin{que}}
\newcommand{\eq}{\end{que}}

\newtheorem*{thm1}{Theorem 1}
\newtheorem*{thmA}{Theorem A}
\newtheorem*{thmB}{Theorem B}
\newtheorem*{rk*}{Remark}
\newtheorem{que}{Question}

\newtheorem{thm}{Theorem}
\newtheorem{lem}{Lemma}

\newtheorem{cor}{Corollary}

\author{Nhat Minh Doan}
\begin{document}
\begin{center}{\Large \bfseries Geometric filling curves on punctured surfaces}\vspace{.2cm}\\
{\large Nhat Minh Doan\symbolfootnote[1]{Research supported by FNR PRIDE15/10949314/GSM. 
}}\end{center}
\allowdisplaybreaks
\begin{abstract}
This paper is about a type of quantitative density of closed geodesics and orthogeodesics on complete finite-area hyperbolic surfaces. The main results are upper bounds on the length of the shortest closed geodesic and the shortest doubly truncated orthogeodesic that are $\varepsilon$-dense on a given compact set on the surface.
\end{abstract}
\section{Introduction}
It is well known that if $X$ is a complete finite-area hyperbolic surface, the set of closed geodesics is dense on $X$ and on the unit tangent bundle of $X$ (see e.g. \cite{beardon2012geometry} or \cite{katok1995introduction}). A type of quantitative density of closed geodesics on closed hyperbolic surfaces was investigated by Basmajian, Parlier, and Souto in \cite{BPS16}. In particular, for any closed hyperbolic surface $X$ and any positive number $\varepsilon$, they found an upper bound on the length of the shortest closed geodesic that is $\varepsilon$-dense on $X$, by which it meant that all points of $X$ are at a distance at most $\varepsilon$ from the geodesic. This upper bound is recently used to estimate the complexity of an algorithm of tightening curves and graphs on surfaces \cite{chang2020tightening}.
Our goal is to extend their results to the case of complete finite-area hyperbolic surfaces in two directions. The first is that for any complete finite-area hyperbolic surface and any positive number $\varepsilon$ less than or equal to $2$, we are going to construct a closed geodesic $\gamma_\varepsilon$ so that $\gamma_\varepsilon$ is $\varepsilon$-dense on a given compact set of the surface and its length is bounded above by a quantity which depends on the geometry of $X$ and $\varepsilon$. The second is that we will construct a doubly truncated orthogeodesic that is $\varepsilon$-dense and also of bounded length. These types of orthogeodesics appear for instance in identities \cite{parlier2020geodesic} related to McShane's identity \cite{mcshane1998simple} and Basmajian's identity \cite{Bas93}.

Let us begin with a few necessary notations so that one can understand the statement of the main results. Let $\mathcal{M}_{g,n}$ be the moduli space of complete connected orientable finite area hyperbolic surfaces of genus $g$ and $n$ cusps. For any $X$ in $\mathcal{M}_{g,n}$ and any positive number $\xi \leq 2$, we define 
$X^\xi$ as a subset of $X$ such that each connected component of the boundary of $X^\xi$ is a horocycle of length $\xi$. A geodesic arc on $X$ is called a doubly truncated orthogeodesic on $X^\xi$ if it is perpendicular to the horocyclic boundary of $X^\xi$ at its endpoints. Our main results are the following.

\begin{thmA}\label{A}
For all $X\in \mathcal{M}_{g,n}$ there exists a constant $C_X>0$ such that for all $0<\xi \leq 1$ and all $0<\varepsilon \leq 2$ there exists a closed geodesic $\gamma_\varepsilon$ that is $\varepsilon$-dense on $X^{\xi}$ and such that 
\begin{center}
$\ell(\gamma_\varepsilon) \leq C_{X} \dfrac{1}{\varepsilon}\bigg(\log \dfrac{1}{\varepsilon}+\log \dfrac{1}{\xi}\bigg)$.
\end{center}
\end{thmA}

\begin{thmB}\label{B}
Let $X\in \mathcal{M}_{g,n}$, there exists a constant $D_{X}>0$ such that for all $0<\xi \leq 1$ and all $0<\varepsilon \leq \min\{2\log\frac{1}{\xi},2\}$ there exists a doubly truncated orthogeodesic $\smallO_\varepsilon$ that is $\varepsilon$-dense on $X^\xi$ and such that 
\begin{center}
$\ell(\smallO_\varepsilon) \leq D_{X} \dfrac{1}{\varepsilon}\bigg(\log \dfrac{1}{\varepsilon}+\log \dfrac{1}{\xi}\bigg)$.
\end{center}
\end{thmB}
Our main ingredient in the proof of Theorem A and Theorem B is the following result.
\begin{thm1}\label{das}
For any $X\in \mathcal{M}_{g,n}$, there exists a constant $K_X$ such that the following holds. For all $0<\varepsilon \leq 1$, $0<\xi \leq 1$ and any finite collection $\{c_i\}^N_{i=1}$ of geodesic arcs of average length $\bar{c}$ in $X^\xi$, there exist a closed geodesic $\gamma$ of length at most 
\begin{center}
$N(K_X+\bar{c}+10\log\frac{1}{\varepsilon}+8\log\frac{1}{\xi})$
\end{center}
containing $\{c_i\}^N_{i=1}$ in its $2\varepsilon$-neighborhood.
\end{thm1}
In the following, we will give a brief outline of the arc-replacement idea in \cite{BPS16} for the case of closed surfaces and explain how we adapt it to the case of punctured surfaces in the proof of Theorem \ref{das}.

Let $X$ be a complete finite-area hyperbolic surface. When $X$ is a closed surface, the steps in the proof of \cite[Theorem 2.4]{BPS16} can be described briefly as follows.
\begin{itemize}
  \item[(i)]  Taking a filling closed geodesic $\gamma_0$ on $X$ which decomposes $X$ into polygons.
  \item[(ii)]  For any $\varepsilon>0$, we take a finite collection $\mathcal{A}_N:=\{c_i\}^N_{i=1}$ of geodesic arcs such that the $\varepsilon$-neighborhood of $\mathcal{A}_N$ covers the whole surface $X$. The number of arcs in this collection is roughly $\frac{1}{\varepsilon}$ up to a constant depending on $X$.
  \item[(iii)]  Extending these arcs in both directions a certain distance $r_{\varepsilon}$ of roughly $\log \frac{1}{\varepsilon}$ and then keep extending them (at most a distance $D$ the diameter of $X$) until they connect to $\gamma_0$ with good angles. 
  \item[(iv)] Constructing a closed piecewise geodesic forming from the collection of extended arcs and suitable subarcs of the filling closed geodesic $\gamma_0$. The resulting closed piecewise geodesic is contained in the $\varepsilon$-neighborhood of the desired closed geodesic $\gamma_{\varepsilon}$. The length of $\gamma_{\varepsilon}$ is bounded above by $C_{X} \frac{1}{\varepsilon}\left(\log \frac{1}{\varepsilon}\right)$ where $C_X$ is a constant depending on $X$.
\end{itemize}
When $X$ is a punctured surface, we identify the main obstruction and propose the key idea in the proof of Theorem \ref{das} step by step as follows.
\begin{itemize}
  \item Similar to steps (i) and (ii) above, we take a filling closed geodesic $\gamma_0$ on $S$ such that $\gamma_0$ cuts $S$ into polygons and once-punctured polygons. For any $\varepsilon>0$, we take a finite collection $\{c_i\}^N_{i=1}$ of geodesic arcs on the truncated surface $S^{\xi}$ which contains $S^{\xi}$ in its $\varepsilon$-neighborhood. 
    \item  Major obstruction: it is almost surely possible to extend these arcs in both directions a certain distance $r_{\varepsilon}$ and then keep extending them until they connect to $\gamma_0$ with good angles but it is not always possible to bound the lengths of these extended arcs since they can go into a cusp region for a long time before connecting to $\gamma_0$.  
     \item Key idea: we replace the collection $\{c_i\}^N_{i=1}$  by a better one in the sense that the new collection, denoted by $\{\zeta_i\}^N_{i=1}$, still contains $\{c_i\}^N_{i=1}$  in its $\varepsilon$-neighborhood and when we extend them, they will not go too deep into the cusp region before connecting to $\gamma_0$ with sufficiently big angles. Denote the collection of these extended arcs by $\{\zeta'_i\}^N_{i=1}$.
     \item Applying the step (iv) above for the collection $\{\zeta'_i\}^N_{i=1}$.
\end{itemize}

\section{Geodesics and horocycles in the hyperbolic plane $\mathbb{H}$ and on surfaces}
In this section, we introduce some elementary properties of geodesics traveling through subsurfaces which we will use to prove Theorem \ref{maintool}. Let $P_n$ be a hyperbolic subsurface with a single polygonal boundary of $n$ concatenated geodesic edges such that all angles are less than $\pi$.
Figure \ref{fig:earing}(a) shows an example of $P_{1}$.  
The following lemma is an extended version of Lemma 1 in \cite{BPS16}.
\begin{lem}\label{goodbad}
There exists $\theta_P>0$ such that any geodesic arc $c$ lying inside $P_n$ with endpoints on edges of the $n$-gons forms an acute angle of at least $\theta_P$ in one of its endpoints. Furthermore, the length of $c$ is at most a constant $\ell_P$ if one of the angles is of value less than or equal to $\theta_P$.
\end{lem}
\begin{proof}
We first label the vertices of the $n$-gonal boundary of $P_n$ by $A_1, A_2, ..., A_n$ consecutively. For each $i\in\{1,2,...,n\}$, we can connect $A_i$ to $A_{i+2}$ by a shortest geodesic arc lying inside the interior of $P_n$ (in which $A_{n+1}:=A_1$ and $A_{n+2}:=A_2$) such that there is no cusp or geodesic boundary component in the resulting triangle $A_iA_{i+1}A_{i+2}$. We call each such resulting triangle to be an ear of $P_n$.
In the set of angles of the ears in $P_n$ (three angles for each ear), we denote by $\theta_P$ their minimum value. Also, in the set of sides of the ears in $P_n$, we denote by $\ell_P$ their maximum value.

\begin{figure}[h!]
    \centering
    \subfloat[An example of $P_1$.]{{\includegraphics[width=5cm]{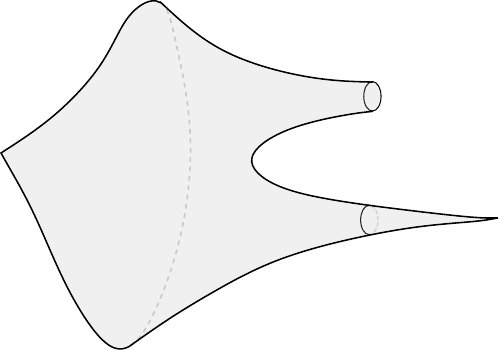} }}
    \qquad
    \subfloat[The geodesic arc $c$ in the triangle $A_1A_2A_3$.]{{\includegraphics[width=5cm]{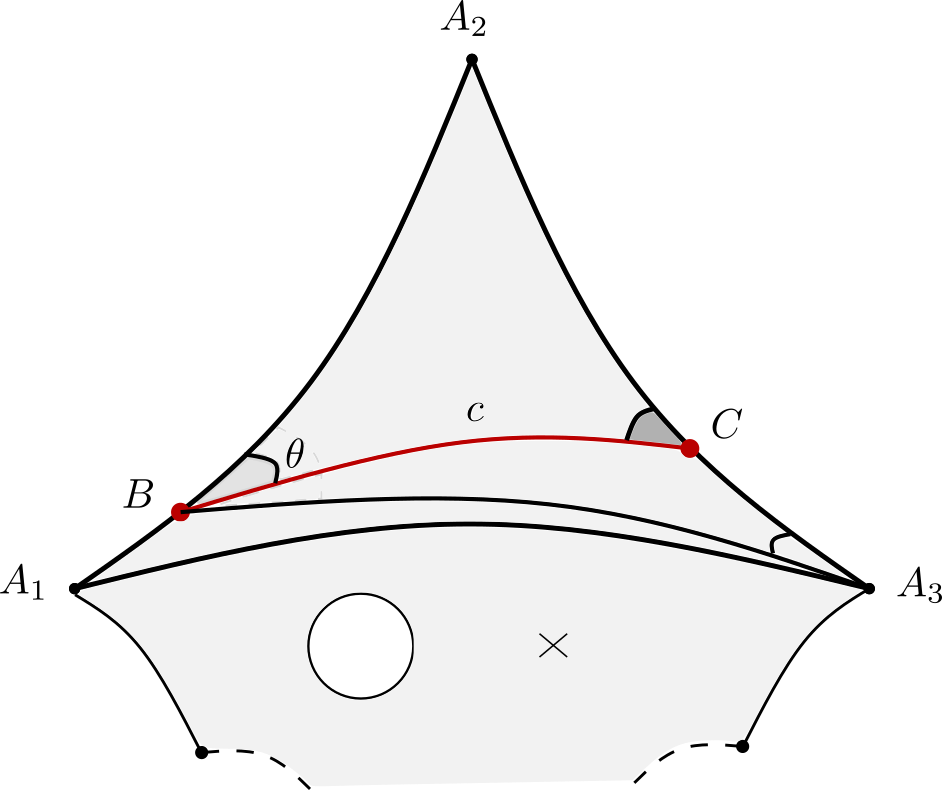} }}
    \caption{}
    \label{fig:earing}
\end{figure}

 Without loss of generality, we can assume that the geodesic arc $c$ leaving from the point $B$ on the interior of a side $A_1A_2$ of the $n$-gons forms an angle $\theta$ of at most $\theta_P$ as in Figure \ref{fig:earing}(b). Since the triangle $BA_2A_3$ is contained in the ear $A_1A_2A_3$, Area$(BA_2A_3)\leq $ Area$(A_1A_2A_3)$. As these two triangles are sharing an angle $A_2$, by Gauss-Bonnet, the sum of the two remaining angles of $BA_2A_3$ is greater than or equal to the sum of the angles $A_1$ and $A_3$ of the ear. In other words, we have 
 $$ \measuredangle  A_2BA_3 + \measuredangle A_2A_3B \geq \measuredangle A_2A_1A_3 + \measuredangle A_2A_3A_1.$$
 
 Since $B$ is on the interior of the side $A_1A_2$, we have $\measuredangle A_2A_3B < \measuredangle A_2A_3A_1$ and thus $\measuredangle A_2BA_3 > \measuredangle A_2A_1A_3$. Hence $\measuredangle A_2BA_3 > \theta_P \geq \theta$ and $c$ lies inside $BA_2A_3$. This implies that $c$ also lies inside the ear and the triangle $BA_2C$ is contained in the ear. By the same argument, we can show that the angle at $C$ (i.e. the remaining angle formed by $c$ and an edge of the $n$-gons) is greater than $\theta_P$. 
The fact $c$ lies inside the ear also tells us that the length of $c$ has to be less than or equal to at least one of three sides of the ear, hence $\ell(c) \leq \ell_P$. Note that $\theta_P$ and $\ell_P$ are optimal as the inner sides of the ears are admissible geodesic arcs.
\end{proof}
In this paper, we only need to focus on the case of once-punctured polygons. We also refer the reader to \cite[Chapter 2]{buser2010geometry} and \cite[Chapter 7]{beardon2012geometry} for all necessary trigonometric formulae. The following lemma will give us an upper bound on the length of the geodesic arc that traverses inside the polygon with endpoints lying on the boundary of the polygon.
\begin{lem}\label{maxtravelling}
Let $P$ be a once-punctured polygon and $\psi \in [0,\frac{\pi}{2})$ a constant. Let $h$ be a closed horocycle lying inside $P$. Let $d$ be the maximal distance from a point on $\partial{P}$ to $h$. Then for any geodesic arc $c$ in $P$ with two endpoints on $\partial P$ and $\theta:=\measuredangle(c,h)\leq \psi$, we have
\begin{center}
$\ell(c) \leq \arccosh\bigg(\dfrac{2e^{2d}}{\cos^2{\psi}}-1\bigg)$.
\end{center}
\end{lem}
\begin{proof}
We lift to $\mathbb{H}$ as in Figure \ref{fig:travelling}.
\begin{center}
\begin{minipage}{\linewidth}
\begin{center}
\includegraphics[width=0.6\linewidth]{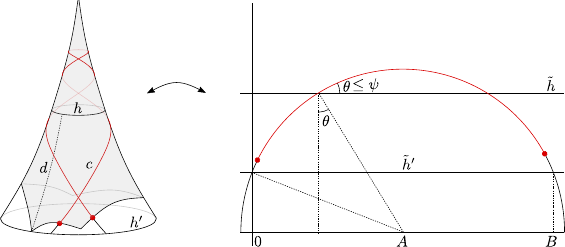}
\captionof{figure}{}
\label{fig:travelling}
\end{center}
\end{minipage}
\end{center}
A lift of $c$ is contained in the geodesic segment with endpoints $i$ and $B+i$, where
\begin{center}
$A^2+1=\bigg(\dfrac{e^d}{\cos\theta}\bigg)^2, \text{and } B=2A.$
\end{center}
Hence 
\begin{center}
$\cosh{\ell(c)}\leq \cosh(d_{\mathbb{H}}(2A+i,i))=1+\dfrac{|2A|^2}{2}=\dfrac{2e^{2d}}{\cos^2\theta}-1\leq \dfrac{2e^{2d}}{\cos^2\psi}-1$.
\end{center}
\end{proof}
The next lemma describes some properties in a certain type of quadrilateral. Let $h_1$ and $h_2$ be two disjoint horocycles in $\mathbb{H}$. Let $A_1A_2$ be the common orthogonal between $h_1$ and $h_2$. For $i=1,2$, let $B_i, C_i$ be points on $h_i$ so that $B_1B_2C_2C_1$ becomes a quadrilateral with two horocyclic edges $\{B_1C_1, B_2C_2\}$ and two geodesic edges $\{B_1B_2, C_1C_2\}$.
\begin{lem}\label{convexity}
Suppose the inner acute angles of the quadrilateral $B_1B_2C_2C_1$ are of the same value $\psi$. Then 
\begin{center}
$\ell(B_1C_1)=\ell(B_2C_2)=2\sqrt{\tan^2\psi+e^{-\ell(A_1A_2)}+1}-2\tan{\psi}$.
\end{center}
Furthermore, every geodesic segment which only meets $h_1$ and $h_2$ at its endpoints lies totally inside the quadrilateral if and only if each of the two acute angles at the endpoints is of value at least $\psi$.
\end{lem}

\begin{figure}[h!]
\centering
  \includegraphics[width=10cm]{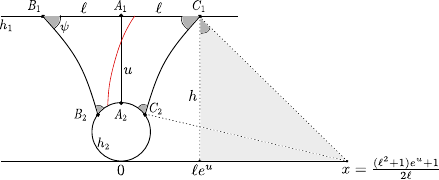}
\captionof{figure}{}
  \label{fig:6}
\end{figure}
\begin{proof} Denote by $u$ the length of $A_1A_2$. Let $h_1$ be the horizontal line $y=ie^u$, $h_2$ be the horocycle centered at 0 and going through $i$ as in Figure \ref{fig:6}. We can also suppose that $A_1=ie^u$, $A_2=i$, hence $C_1=\ell e^u +ie^u$ where $\ell$ is defined by the length of the horocyclic segment $A_1C_1$. By symmetry of the quadrilateral, we can find an involution $f$ which is a non-orientation-preserving isometry sending $A_1$ to $A_2$, $B_1$ to $B_2$ and $C_1$ to $C_2$. By a standard computation, $f(z)=\frac{a\bar{z}+b}{c\bar{z}+d}$ where $a=d=0, b=e^{\frac{u}{2}}$, and $c=e^{\frac{-u}{2}}$. As a consequence,
\begin{center}
 $C_2=f(C_1)=f(\ell e^u +ie^u)=\dfrac{\ell}{\ell^2+1}+\dfrac{i}{\ell^2+1}$.
\end{center}
Let $x$ be a point on the real line of $\mathbb{H}$ such that the Euclidean distances from $C_1$ and $C_2$ to $x$ are the same. By computation, $x=\frac{(\ell^2+1)e^u+1}{2\ell}$. Now applying the Euclidean trigonometric formula for the shaded Euclidean right triangle in figure \ref{fig:6}, noting that the value of the angle at $C_1$ of this triangle is exactly $\psi$, we get
\begin{center}
$\tan \psi=\dfrac{\frac{(\ell^2+1)e^u+1}{2\ell}-\ell e^u}{e^u}=\dfrac{-\ell^2+1+e^{-u}}{2\ell}.$
\end{center}
From this, we obtain the value of $\ell$ in terms of $\psi$ and $u$.

For the second part, we fix an angle $\phi$ of value between $\psi$ and $\frac{\pi}{2}$ at one endpoint of $c$, and observe what happens to the acute angle at the other endpoint of $c$ while moving $c$ along the horocycles and keeping the value of the angle $\phi$. The behavior of the values of the remaining acute angle is exactly that of a concave function. By symmetry of the quadrilateral, $c$ lies inside the quadrilateral if and only if both acute angles at the endpoints are of value at least $\psi$.
\end{proof}
 Next, we recall two useful facts. 
\begin{lem}\cite[Lemmas 2.2]{BPS16} \label{disjoint}
Let $\dfrac{\pi}{2}\geq \theta_0 >0$, and set $m(\theta_0):=2\log\bigg(\dfrac{1}{\sin{\theta_0}}\bigg)+2\log(1+\cos{\theta_0})$. If $c$ is an oriented geodesic segment in $\mathbb{H}$ of length at least $m(\theta_0)$ between two (complete) geodesics $\gamma_1,\gamma_2$ such that the starting (resp. end) point of $c$ lies on $\gamma_1$ (resp. $\gamma_2$) and $\measuredangle(c,\gamma_i)\geq \theta_0$ for $i=1,2$, then $\gamma_1$ and $\gamma_2$ are disjoint.
\end{lem}
\begin{lem}\cite[Lemmas 2.3]{BPS16}\label{epsilondistance}
Let  $\dfrac{\pi}{2}\geq \theta_0 >0$ be a fixed constant. Let $c$ be a geodesic segment in $\mathbb{H}$ and $\gamma_c$ the complete geodesic containing $c$. Fix $\varepsilon \in (0,2]$ and let $\gamma_1$ and $\gamma_2$ be geodesics that intersect $\gamma_c$ such that intersection points $p_1, p_2$ lie on different sides of $c$. There exists an $r_{\varepsilon}$ (depending only on $\varepsilon$ and $\theta_0$), so that if $\measuredangle(\gamma_i,\gamma_c)\geq \theta_0$ and $d(c,p_i)\geq r_{\varepsilon}$, for $i=1,2$, then  $\gamma_1$ and $\gamma_2$ are disjoint.\\ 
Furthermore, for any geodesic $\gamma$ intersecting both $\gamma_1$ and $\gamma_2$, we have the following properties:

(P1) $c\subset B_{\varepsilon}(\gamma)$.

(P2)  The image of the orthogonal projection of $c$ on $\gamma$ is contained in the middle part of $\gamma$ (i.e. it lies between $\gamma_1$ and $\gamma_2$).

\end{lem} 
\begin{proof}
The properties: ``$\gamma_1$ and $\gamma_2$ are disjoint'' and (P1) are proved in \cite[Lemma 2.3]{BPS16}, here we will fix a minor mistake in their proof to obtain a correct value of $r_{\varepsilon}$ under the requirement that $0<\varepsilon \leq 2$. 

\textbf{(P1)} Keeping all notations introduced in \cite[Lemma 2.3]{BPS16}, from the proof we already had:
\begin{equation}\label{neww}
\cosh{\frac{\ell(\mu)}{2}}=\cosh\bigg(r_\varepsilon+\frac{\ell(c)}{2}\bigg)\sin(\theta_0) \,\,\,;\,\,\,  \sinh{d'}=\dfrac{1}{\sinh{\frac{\ell(\mu)}{2}}} \,\,\,;\,\,\,  \sinh{h'}=\sinh{d'}\cosh {\frac{\ell(c)}{2}}.
\end{equation}
In \cite{BPS16}, the authors deduced from Equalities \ref{neww} the following incorrect relation:
\begin{center}
$\sin(\theta_0)\cosh\bigg(r_\varepsilon+\frac{\ell(c)}{2}\bigg)\sinh{h'}=\cosh {\frac{\ell(c)}{2}}$
\end{center}
At the end of their proof, they deduced the following relation
\begin{center}
$r_{\varepsilon}\geq \log \bigg(\dfrac{1}{\varepsilon}\bigg)+\log{ \bigg(\dfrac{4}{\sin{\theta_0}}\bigg)}$,
\end{center}
which holds for any $\varepsilon >0$. Then they chose $r_{\varepsilon}:= \log \bigg(\dfrac{1}{\varepsilon}\bigg)+\log{ \bigg(\dfrac{4}{\sin{\theta_0}}\bigg)}$. Fortunately, this incorrect value of $r_{\varepsilon}$ does not affect the overall conclusion of the main theorems since later on one will see that the condition $0<\varepsilon \leq 2$ is necessary, and then the correct value of  $r_{\varepsilon}$ can be chosen as $\log \bigg(\dfrac{1}{\varepsilon}\bigg)+\log{ \bigg(\dfrac{2e}{\sin{\theta_0}}\bigg)}$. From Equalities \ref{neww}, we deduce that
\begin{center}
$\sinh^2(h')=\sinh^2(d')\cosh^2\bigg(\frac{\ell(c)}{2}\bigg)=\dfrac{\cosh^2\bigg(\frac{\ell(c)}{2}\bigg)}{\cosh^2\bigg(\frac{\ell(\mu)}{2}\bigg)-1}=\dfrac{\cosh^2\bigg(\frac{\ell(c)}{2}\bigg)}{\cosh^2\bigg(r_\varepsilon+\frac{\ell(c)}{2}\bigg)\sin^2(\theta_0)-1}$.
\end{center}
We want to show that $h' \leq \frac{\varepsilon}{2}$ thus that
\begin{equation}\label{new1}
\dfrac{\cosh^2\bigg(\frac{\ell(c)}{2}\bigg)}{\cosh^2\bigg(r_\varepsilon+\frac{\ell(c)}{2}\bigg)\sin^2(\theta_0)-1} \leq \sinh^2 \bigg(\dfrac{\varepsilon}{2}\bigg)
\end{equation}
and Inequality \ref{new1} is equivalent to the following:
\begin{center}
$\cosh^2\bigg(r_\varepsilon+\frac{\ell(c)}{2}\bigg) \geq \dfrac{\cosh^2\big(\frac{\ell(c)}{2}\big)+\sinh^2 \big(\frac{\varepsilon}{2}\big)}{\sinh^2 \big(\frac{\varepsilon}{2}\big)\sin^2(\theta_0)}$.
\end{center}
By using the identities  $\cosh(2x)=2\cosh^2(x)-1=2\sinh^2(x)+1$, the last inequality can be expressed differently as follows:
\begin{equation}\label{new2}
2r_\varepsilon \geq \arccosh \bigg( \dfrac{\cosh{\ell(c)}+\cosh{\varepsilon}}{\sinh^2 \big(\frac{\varepsilon}{2}\big)\sin^2(\theta_0)}-1 \bigg) -\ell(c).
\end{equation}
The right hand of Inequality \ref{new2} can be considered as a function:
\begin{center} $f(x)=\arccosh(ax+b)-\arccosh{x}$ \end{center}
on the domain $[1,\infty)$, in which $$a=\frac{1}{\sinh^2\left(\frac{\varepsilon}{2}\right)\sin^2(\theta_0)}>0,$$ 
$$b=\frac{\cosh\varepsilon}{\sinh^2\left(\frac{\varepsilon}{2}\right)\sin^2(\theta_0)}-1\geq\frac{\cosh\varepsilon}{\sinh^2\left(\frac{\varepsilon}{2}\right)}-1=1+\frac{1}{\sinh^2\left(\frac{\varepsilon}{2}\right)}>1.$$This function reaches its maximum $x=1$. Hence Inequality \ref{new2} will hold if
\begin{equation}\label{new3}
2r_\varepsilon \geq \arccosh \bigg( \dfrac{1+\cosh{\varepsilon}}{\sinh^2 \big(\frac{\varepsilon}{2}\big)\sin^2(\theta_0)}-1 \bigg).
\end{equation}
For simplicity, we set $A:=\dfrac{1+\cosh{\varepsilon}}{\sinh^2 \big(\frac{\varepsilon}{2}\big)\sin^2(\theta_0)}$. Note that
\begin{center}
$\arccosh(A-1) < \arccosh A = \log (A+\sqrt{A^2-1})<\log(2A)$
\end{center}
and
\begin{center}
$\log(2A)=\log \bigg( \dfrac{2+2\cosh{\varepsilon}}{\sinh^2 \big(\frac{\varepsilon}{2}\big)\sin^2(\theta_0)} \bigg)=2\log{ \bigg(\dfrac{2}{\sin{\theta_0}}\bigg)}+2\log{\bigg(\dfrac{e^\varepsilon+1}{e^\varepsilon-1}\bigg)}$
\end{center}
in which $\log{\bigg(\dfrac{e^\varepsilon+1}{e^\varepsilon-1}\bigg)}<\log{ \bigg(\dfrac{1}{\varepsilon}\bigg)}+1$ for all $\varepsilon \in (0, 2]$. \\
Thus Inequality \ref{new3} certainly holds provided
\begin{center}
$r_\varepsilon \geq \log{ \bigg(\dfrac{2}{\sin{\theta_0}}\bigg)}+\log{ \bigg(\dfrac{1}{\varepsilon}\bigg)}+1$.
\end{center} 
Hence we set 
\begin{center}
$r_{\varepsilon}:=\log \bigg(\dfrac{1}{\varepsilon}\bigg)+\log{ \bigg(\dfrac{2e}{\sin{\theta_0}}\bigg)}$.
\end{center}
\textbf{(P2)} We consider the worst case scenario: $c$ is a complementary $\theta_0$-transversal of $\gamma_1$ and $\gamma_2$ (see Figure \ref{hh}). Consider the limit case which is when $\gamma$ and $\gamma_2$ are ultra-parallel. Now we orient $c$ from $\gamma_1$ to $\gamma_2$. Let $\psi$ denote the angle between $\gamma$ and the geodesic segment connecting the endpoint of $\gamma$ on $\gamma_1$ and the starting point of the oriented geodesic segment $c$. Let $\theta$ denote the angle between the extended part of $c$ toward $\gamma_2$ and the geodesic ray starting at the endpoint of $c$ and ending at the endpoint of $\gamma$ at infinity. Notice that the image of the orthogonal projection of $c$ on $\gamma$ lies between $\gamma_1$ and $\gamma_2$ if and only if the angle $\psi$ is acute. Since the sum of four inner angles in a quadrilateral is always less than $2\pi$,  $\psi \leq \frac{\pi}{2}$ holds if $\theta <\frac{\pi}{4}$. By using the same formula as in the proof of Lemma \ref{disjoint}, we have:
\begin{center}
$\cos{\theta}=\dfrac{\tanh{r_{\varepsilon}}-\cos{\theta_0}}{1-\tanh{r_{\varepsilon}}\cos{\theta_0}}$.
\end{center}
Hence $\theta <\frac{\pi}{4}$ holds provided
\begin{equation}\label{a}
\dfrac{\tanh{r_{\varepsilon}}-\cos{\theta_0}}{1-\tanh{r_{\varepsilon}}\cos{\theta_0}}>\dfrac{1}{\sqrt{2}}.
\end{equation}
\begin{center}
\begin{minipage}{\linewidth}
\begin{center}
\includegraphics[width=0.25\linewidth]{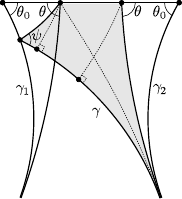}
\captionof{figure}{The worst case scenario.}
\label{hh}
\end{center}
\end{minipage}
\end{center}
By a small manipulation, Inequality \ref{a} is equivalent to the following:
\begin{center}
$r_{\varepsilon}>\dfrac{1}{2}\log\bigg(\dfrac{1}{\sin\theta_0}\bigg)+\log(1+\sqrt{2})+\log(1+\cos\theta_0)$.
\end{center}
And this last inequality holds by definition of $r_{\varepsilon}$.
\end{proof}
\section{Main tools}
Moduli space $\mathcal{M}_{g,n}$ we think of as the space of complete hyperbolic structures up to isometry on a punctured orientable topological surface $\Sigma_{g,n}$ of genus $g$ with $n$ punctures (with $2g+n\geq 3$). A cusp region of area $\xi$ is a portion of the surface isometric to $\{z: \text{Im} z \geq 1\}/{z\mapsto z+\xi}$. For  any $X$ in $\mathcal{M}_{g,n}$ and any positive number $\xi \leq 2$, we can define
 \begin{center}
$X^\xi:=\text{cl}(X\setminus \{\text{all cusp regions of area } \xi\})$.
\end{center}
In another word, $X^\xi$ is a surface of genus $g$ with $n$ boundary components and each connected component of its boundary is a horocycle of length $\xi$.
The following theorem is the main technical part of this paper:
\begin{thm}\label{maintool}
For any $X\in \mathcal{M}_{g,n}$, there exists a constant $K_X$ such that the following holds. For all $0<\varepsilon \leq 1$, $0<\xi \leq 1$ and any finite collection $\{c_i\}^N_{i=1}$ of geodesic arcs of average length $\bar{c}$ in $X^\xi$, there exists a closed geodesic $\gamma$ of length at most 
\begin{center}
$N(K_X+\bar{c}+10\log\frac{1}{\varepsilon}+8\log\frac{1}{\xi})$
\end{center}
which contains $\{c_i\}^N_{i=1}$ in its $2\varepsilon$-neighborhood.
\end{thm}
\begin{proof} The proof is structured in three parts. The first part will introduce some necessary geometric quantities and a classification of geodesics traveling through (punctured) polygons on the surface $X$ forming big or small angles with the sides of the polygons. The second part is about the details of the arc-replacement technique and some upper bounds of the length of extended arcs. The final part is a recap of parts 1 and 2 followed by the construction of the closed geodesic $\gamma$. 

\textbf{Part 1: Setup.}

We take a closed geodesic $\gamma_0$ on $X$ of minimal length such that $X \setminus \gamma_0$ consists of a finite collection of ordinary polygons $\{P_i\}_{i\in I}$ and once-punctured polygons $\{P_i\}_{i \in J}$ ($I$ and $J$ are two disjoint finite index sets). Recall that, for each polygon $P_i$ ($i\in I\cup J$) we have the constants $\theta_{P_i}$ and $\ell_{P_i}$ as mentioned in Lemma \ref{goodbad}. Also in each ordinary polygon $P_i$, we denote by $D_{P_i}$ the value of its intrinsic diameter. Note that there is no intrinsic diameter in once-punctured polygons. We define:
\begin{center}
$\theta_0:=\min\limits_{i\in I\cup J} \{ \theta_{P_i} \}$ and $D:=\max\limits_{i\in I, j\in J} \{D_{P_i}, \ell_{P_j}\}$.

\end{center} 
In this part, we aim to define a classification for geodesics traveling through polygons in the following way.

We begin by defining a closed horocycle that lies inside a once-punctured polygon (hence $\gamma_0$ and this horocycle have no intersection) such that the distance between this horocycle and the horocycle of length $\xi$ (namely $h_{\xi}$) is at least 
\begin{center}
$r_{\varepsilon}:=\log(\frac{1}{\varepsilon})+\log \big(\frac{2e}{\sin\theta_0}\big)$.
\end{center}
Since $\gamma_0$ wraps around each cusp at most once, it will not cross transversely the horocycle of length $1$ in each cusp region. We also note that $\xi$ is less than $1$. Hence, one option which satisfies the above condition is the horocycle which is at distance $r_{\varepsilon}$ from $h_{\xi}$. We denote this horocycle by $h$. 
Since the decay of the length of horocycle in a cusp is $e$, 
\begin{center}
$\ell(h)=\dfrac{\xi}{e^{r_\varepsilon}}=\dfrac{\varepsilon\xi}{2e}\sin\theta_0$.
\end{center}
Now, let $c$ be an arbitrary geodesic arc on $X$, we extend  $c$ by $r_\varepsilon$ in one direction to get a new arc $c'$ and the new endpoint $p'$. Then we continue to extend $c'$ from $p'$. In the process of extending, the geodesic can intersect $\gamma_0$ several times and form angles. An intersection is called a \textbf{good intersection} if the acute angle at it is at least $\theta_0$, and otherwise, it will be called a \textbf{bad intersection}. The extension will stop at the first good intersection from $p'$. By Lemma \ref{goodbad}, the extensions can be divided into 5 cases as follows:

1. From $p'$, the previous intersection is bad and the next intersection is good.
\begin{center}
\begin{minipage}{\linewidth}
\begin{center}
\includegraphics[width=0.5\linewidth]{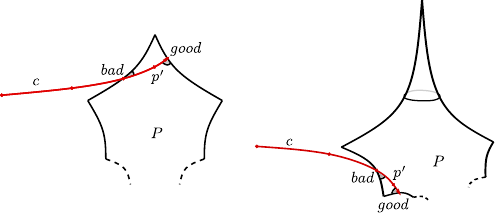}
\captionof{figure}{ Case 1.}
\end{center}
\end{minipage}
\end{center}

2. From $p'$, the previous intersection is good, and the next intersection is bad.
\begin{center}
\begin{minipage}{\linewidth}
\begin{center}
\includegraphics[width=0.4\linewidth]{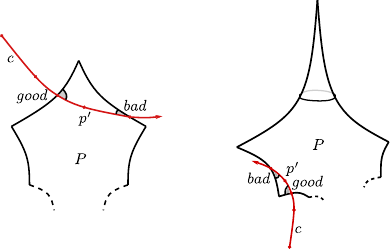}
\captionof{figure}{Case 2. }
\end{center}
\end{minipage}
\end{center}

3. $p'$ lies inside an ordinary polygon, the previous intersection and the next intersection are both good (see Figure \ref{2}).

4. $p'$ lies inside a once-punctured polygon $P$, the previous intersection and the next intersection are both good (see Figure \ref{2}) and so that the geodesic arc, namely $c''$, between these two intersections is not too long, more precisely, this arc either intersects the horocycle $h$ at an angle less than a given angle $\psi$ or does not intersect $h$.

5. $p'$ lies inside a once-punctured polygon and if we continue to extend $c'$ from $p'$, it will intersect the horocycle $h$ at an angle at least $\psi$ (see Figure \ref{2}).
\begin{center}
\begin{minipage}{\linewidth}
\begin{center}
\includegraphics[width=0.6\linewidth]{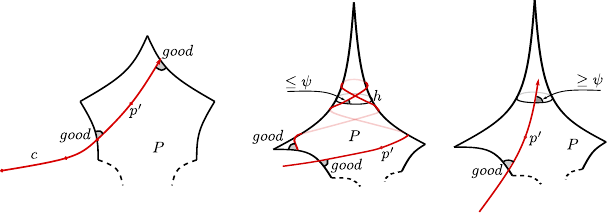}
\captionof{figure}{Cases 3,4 and 5 respectively.}
\label{2}
\end{center}
\end{minipage}

\end{center}
Now in order to stop the extension, by Lemma \ref{goodbad} and the definition of $D$ above, the distance we need to extend from $p'$ is at most $D$ (in cases 1 and 3), and $2D$ (in case 2). In case 4, let $s_0\in \partial{P}$ such that 
\begin{center}
$d_{X}(s_0, h)=\max \limits_{s\in \partial{P}} \{ d_{X}(s, h) \}$.
\end{center}
Let $d_P$ be the distance from $s_0$ to the closed horocycle of length 1 of the same cusp. Note that the distance between this horocycle and $h$ is $\log\left(\frac{2e}{\varepsilon \xi \sin\theta_0}\right)$. Thus

\begin{center}
$d_{X}(s_0, h) = d_P+\log\left(\dfrac{2e}{\varepsilon \xi \sin\theta_0}\right)$.
\end{center}
Then applying Lemma \ref{maxtravelling} and the inequality $\arccosh(x-1)<\log(2x)$ we have
\begin{equation}\label{c''}
\ell(c'') \leq \arccosh\bigg(\dfrac{2e^{2d_P+2\log\left(\frac{2e}{\varepsilon \xi \sin\theta_0}\right)}}{\cos^2{\psi}}-1\bigg)<2d_P+2\log\left(\frac{2e}{\varepsilon \xi \sin\theta_0}\right)+2\log\bigg( \dfrac{2}{\cos\psi}\bigg).
\end{equation}
Note that, in part 2, we will define $\psi$ as the angle formed by $\tilde{h}$ and $\tilde{\eta}_1$ (see Figure \ref{caseB}). By simple computations, we obtain
\begin{center} $\psi =\arccos \left(\dfrac{\frac{\varepsilon\xi}{e}\sin\theta_0}{1+\frac{\varepsilon^2\xi^2}{4e^2}\sin^2(\theta_0)}\right)$.
\end{center}

From there we have
\begin{equation}\label{sdf}
2\log\bigg( \dfrac{2}{\cos\psi}\bigg)=2\log\bigg(1+\dfrac{\varepsilon^2\xi^2}{4e^2}\sin^2(\theta_0) \bigg)+2\log\bigg(\dfrac{2e}{\varepsilon\xi\sin\theta_0}\bigg)<2+2\log\bigg(\dfrac{2e}{\varepsilon\xi\sin\theta_0}\bigg).
\end{equation}
Combining \ref{c''} and \ref{sdf}, one has
\begin{center}
$\ell(c'')<2d_P+4\log\bigg(\dfrac{2e}{\varepsilon\xi\sin\theta_0}\bigg)+2$.
\end{center}
Moreover, if one denotes by $d_{\gamma_0}$ the maximal value of $\{d_{P_i}\}_{i\in J}$, then
\begin{center}$m_A:=2D+2d_{\gamma_0}+2+4\log\bigg(\dfrac{2e}{\varepsilon\xi\sin\theta_0}\bigg)$
\end{center} 
is an upper bound on the length of the extension in \textbf{class A} (i.e., cases 1, 2, 3, and 4). Note that, in \textbf{class B} (i.e., case 5), the length of the extension is unbounded when $\measuredangle(c',h)$ goes to $\frac{\pi}{2}$.

\textbf{Part 2: Replacements and estimates.}

Now, let $c$ be an arbitrary geodesic arc in the collection $\{c_i\}^N_{i=1}$. Denote the two endpoints of $c$ by $p$ and $q$. We extend $c$ by $r_\varepsilon$ in both directions to get a new geodesic arc $c'$. We now look at different cases.

\underline{Case A}: The extensions in both directions are in class A. 

In order to get good intersections in both directions, we need to extend $c'$ by at most $m_A$ for each of its directions. Hence an upper bound on the length of $c$ after being extended is
\begin{center}
$\ell(c)+2r_\varepsilon+2m_A$
\end{center}
or more precisely,
\begin{equation}\label{b}
\ell(c)+10\log\dfrac{1}{\varepsilon}+8\log \dfrac{1}{\xi}+4D+4d_{\gamma_0}+4+10\log \bigg(\dfrac{2e}{\sin\theta_0}\bigg).
\end{equation}

\underline{Case B}: There is a direction where the extension is in class B. 

Let $p'$ be the endpoint of $c'$ in this direction, we can suppose $p'$ lies inside a once-punctured polygon, namely $P$.

What we aim to do is to replace $c$ with another geodesic arc, which is very close to $c$ and controlled in both directions (i.e. the extension in each direction is in class A). Denote the complete geodesic containing $c$ by $\gamma_c$. We assume that the horizontal line  $y=\frac{2e}{\varepsilon \xi \sin\theta_0}i$, namely $\tilde{h}$, is a lift of the closed horocycle $h$ of length $\frac{\varepsilon \xi}{2e}\sin\theta_0$ in $P$. From there we can suppose the complete geodesic $\tilde{\gamma_c}$ with an endpoint at $0$, forming an angle at least $\psi$ with  $\tilde{h}$, is a lift of $\gamma_c$. Hence $\tilde{\gamma_c}$  lies between $\tilde{\eta}_1$ and $\tilde{\eta}_2$, in which $\tilde{\eta}_1$ and $\tilde{\eta}_2$ are the complete geodesics with a common endpoint at 0, containing $\frac{2e}{\varepsilon \xi \sin\theta_0}i+1$ and $\frac{2e}{\varepsilon \xi \sin\theta_0}i-1$, respectively.

Now we construct lifts of other points from there. Let $\tilde{p}$ and $\tilde{q}$ be lifts of $p$ and $q$, respectively (see Figure \ref{caseB}). Let $\tilde{\gamma}_1$ and $\tilde{\gamma}_2$ be geodesics going through $\tilde{p}$ and $\tilde{q}$, respectively, and orthogonal to the axis $x=0$. Let $\tilde{p}_i:=\tilde{\gamma}_1 \cap \tilde{\eta}_i$ and $\tilde{q}_i:=\tilde{\gamma}_2 \cap \tilde{\eta}_i$, for $i=1,2$.

For  $i=1,2$, we denote by $p_iq_i$ the projection of $\tilde{p_i}\tilde{q_i}$ to the surface $X$. Extend the geodesic arc $p_iq_{i}$ by $r_\varepsilon$ in both directions to get a new geodesic arc $p'_iq'_{i}$. If $p'_i \notin P$, we only need to extend by an extra at most $\frac{\varepsilon \xi}{e}\sin\theta_0+2\ell(\gamma_0)$ to get into $P$. This can be proved by using the inequality in the triangle formed by $\tilde{\eta}_i$, the geodesic segment $\tilde{p_i}'\tilde{p}'$ and a lift of $\gamma_0$ (one of the boundary components of the shaded part in Figure \ref{caseB}).
\begin{center}
\begin{minipage}{\linewidth}
\begin{center}
\includegraphics[width=0.7\linewidth]{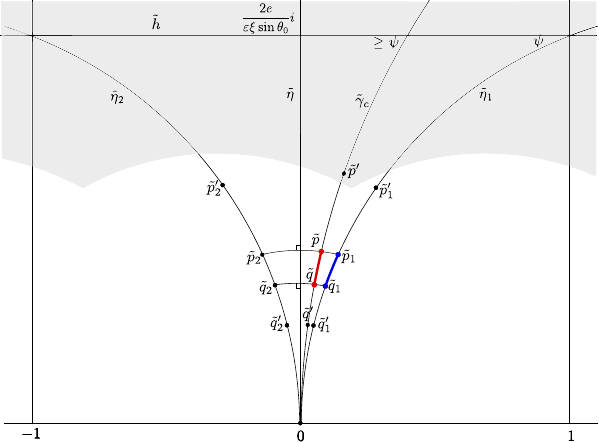}
\captionof{figure}{Lifting to $\mathbb{H}$ in Case B. The shaded part is a lift of polygon $P$.}
\label{caseB}
\end{center}
\end{minipage}
\end{center}
There are two different sub-cases of case B:

\underline{Sub-case BA}: The extension in the direction of either $q_1$ or $q_2$ is in class A.

Without loss of generality, we assume that the extension in the direction of $q_1$ is in class A. Note that the geodesic segments $\tilde{p}\tilde{p}_1$ and $\tilde{q}\tilde{q}_1$ are of length at most $\frac{\varepsilon \xi}{e}{\sin\theta_0}$. Since $\xi \leq 1$ and $\sin\theta_0 \leq 1$, $\frac{\varepsilon \xi}{e}{\sin\theta_0}< \frac{\varepsilon}{2}$. Thus any geodesic containing $p_1q_1$ in its $\varepsilon$-neighborhood contains $c$ in its $\big(\frac{\varepsilon}{2}+\varepsilon\big)$-neighborhood. Hence, in this case, we will replace $c$ with $p_1q_1$.

Recall that we extended $p_1q_1$ by $r_{\varepsilon}$ in both directions to obtain the geodesic arc $p'_1q'_1$. In order to get good intersections in both directions, we continue to extend $p'_1q'_1$ by at most 
\begin{center}
$2m_A+\dfrac{\varepsilon \xi}{e}{\sin\theta_0}+2\ell(\gamma_0)$ .
\end{center}

Hence an upper bound on the length of $p_1q_1$ after being extended is
\begin{center}
$\ell(c)+2r_\varepsilon+2m_A+\dfrac{\varepsilon \xi}{e}{\sin\theta_0}+2\ell(\gamma_0)$ 
\end{center}
which is less than or equal to
\begin{equation}\label{c}
\ell(c)+10\log\dfrac{1}{\varepsilon}+8\log \dfrac{1}{\xi}+4D+4d_{\gamma_0}+4+10\log \bigg(\dfrac{2e}{\sin\theta_0}\bigg)+\dfrac{\sin\theta_0}{e}+2\ell(\gamma_0).
\end{equation}
\underline{Sub-case BB}: The extensions in the directions of $q_1$ and $q_2$ are both in class B. 

Since $\eta, \eta_1$, and $\eta_2$  asymptotic in the direction of $q, q_1$ and $q_2$, length of the geodesic arc orthogonal to $\eta$ and connecting $q'_1$ and $q'_2$ is very small, roughly less than $\frac{\varepsilon \xi}{{2e}}{\sin\theta_0}$. Thus we can suppose that $q'_1$ and $q'_2$ lie in the same once-punctured polygon, denoted by $P'$. Let $h'$ be the closed horocycle of length $\frac{\varepsilon \xi}{{2e}}{\sin\theta_0}$ in $P'$. From there we construct a lift of $h'$, denoted by $\tilde{h'}$. We keep all the notations $A_1, A_2, B_1, B_2, C_1, C_2$ and $u$ as introduced in Lemma \ref{convexity} (see Figure \ref{caseBB} below). 
\begin{center}
\begin{minipage}{\linewidth}
\begin{center}
\includegraphics[width=0.8\linewidth]{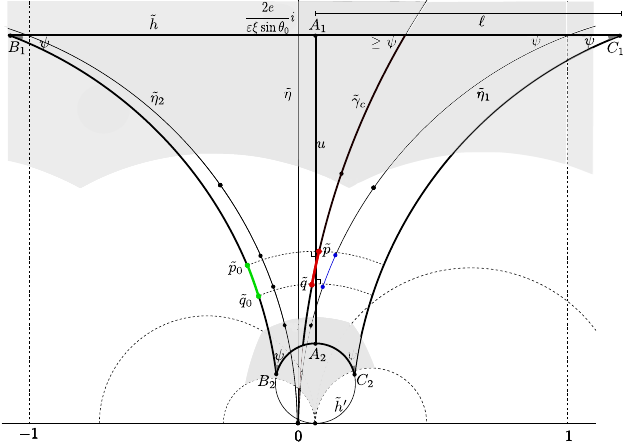}
\captionof{figure}{Lifting to $\mathbb{H}$ in sub-case BB.}
\label{caseBB}
\end{center}
\end{minipage}

\end{center}
Now we would like to apply Lemma \ref{convexity} to the two horocycles $\tilde{h}$ and $\tilde{h'}$ with the angle $\psi$. Recall that $u=\ell(A_1A_2)$ is the distance from $\tilde{h}$ to $\tilde{h'}$. One can estimate a lower bound and an upper bound on $u$ as follows: 
\begin{center}
$2\log\dfrac{{4e}}{\varepsilon\xi{\sin\theta_0}} \leq u\leq \ell(c)+2r_{\varepsilon}+ \dfrac{\varepsilon \xi}{{e}}{\sin\theta_0}+2\ell(\gamma_0)+2\bigg(d_{\gamma_0}+\log{\dfrac{{2e}}{\varepsilon \xi{\sin\theta_0}}}\bigg).$
\end{center}
This implies that:
\begin{center}
$0< u\leq \ell(c)+4\log{\dfrac{1}{\varepsilon}}+2\log{\dfrac{1}{\xi}}+k_X$
\end{center}
in which $k_X:=2d_{\gamma_0}+4\log \frac{2e}{\sin\theta_0}+\frac{\sin\theta_0}{e}+2\ell(\gamma_0)$. 

Recall that $\ell:=\frac{\ell(B_1C_1)}{2}$. Then by Lemma \ref{convexity}, we have 
\begin{center}
$\ell=\sqrt{\tan^2\psi+e^{-u}+1}-\tan{\psi}=\dfrac{e^{-u}+1}{\sqrt{\tan^2\psi+e^{-u}+1}+\tan{\psi}}< \dfrac{2}{\sqrt{\tan^2\psi+1}+\tan{\psi}}$
\end{center}
Thus since $\psi =\arccos \left(\dfrac{\frac{\varepsilon\xi}{e}\sin\theta_0}{1+\frac{\varepsilon^2\xi^2}{4e^2}\sin^2(\theta_0)}\right)$, $0<\varepsilon \leq 1$, $0<\xi \leq 1$ and $\frac{2e}{\sin\theta_0} >2e$ we have
\begin{center}
$\ell < \dfrac{2}{\sqrt{\tan^2\psi+1}+\tan{\psi}}=\dfrac{2\varepsilon \xi}{\frac{2e}{\sin\theta_0}}<2e^{-1}$,
\end{center} 
hence
\begin{center}
$\ell(B_1B_2)=\ell(C_1C_2)\leq 2\ell+u < 4e^{-1}+\ell(c)+4\log\dfrac{1}{\varepsilon}+2\log\dfrac{1}{\xi}+k_X$.
\end{center}
Now we draw a geodesic going through $\tilde{p}$, resp. $\tilde{q}$, orthogonal to $A_1A_2$, meeting $B_1B_2$ at a point, denoted by $\tilde{p}_0$, resp. $\tilde{q}_0$ (see the Figure \ref{caseBB}). We also denote by $\tilde{\zeta}$ the geodesic segment  $\tilde{p}_0\tilde{q}_0$. By projecting the geodesic segment $\tilde{\zeta}$ to $X$, we get a geodesic arc on $X$, denoted by $\zeta$. Note that the geodesic segments $\tilde{p}\tilde{p}_0$ and $\tilde{q}\tilde{q}_0$ are of length at most $2\ell<\frac{4\varepsilon \xi \sin\theta_0}{2e}<\varepsilon$. Thus any geodesic containing $\zeta$ in its $\varepsilon$-neighborhood contains $c$ in its $2\varepsilon$-neighborhood. In this case, we will replace $c$ with $\zeta$.

 Since $B_1B_2$ contains $\tilde{\zeta}$, we will extend $B_1B_2$ instead of $\tilde{\zeta}$.  By Lemma \ref{maxtravelling}, in order to get good intersections in both directions, we need to extend $B_1B_2$ in each direction by a distance $v$, where $v$ satisfies:
\begin{center}
$\log\left(\dfrac{2e}{\varepsilon \xi \sin\theta_0}\right)+\log\dfrac{1+\sin\psi}{1-\sin\psi}\leq v\leq \dfrac{1}{2}\arccosh \left(\dfrac{2e^{2d}}{\cos^2\psi}-1\right)+\dfrac{1}{2}\log\dfrac{1+\sin\psi}{1-\sin\psi}$
\end{center}
where $d:=d_{\gamma_0}+\log\left(\dfrac{2e}{\varepsilon \xi \sin\theta_0}\right)$. Similarly to \ref{c''} and \ref{sdf}, one can show that:
\begin{center}
$\log\dfrac{1}{\varepsilon}+\log\left(\dfrac{2e}{\xi \sin\theta_0}\right)+\log\dfrac{1+\sin\psi}{1-\sin\psi}\leq v < d_{\gamma_0}+1+3\log\dfrac{1}{\xi}+3\log\dfrac{1}{\varepsilon}+3\log\left(\dfrac{2e}{ \sin\theta_0}\right)$
\end{center}
Since the lower bound of $v$ is greater than $r_{\varepsilon}$, we do not need to extend the segment $B_1B_2$ in two steps as in the previous cases.

In this way, we obtain an upper bound on the length of $B_1B_2$ after the extension:
\begin{center}
$4e^{-1}+\ell(c)+4\log\dfrac{1}{\varepsilon}+2\log\dfrac{1}{\xi}+k_X+2d_{\gamma_0}+2+6\log\dfrac{1}{\xi}+6\log\dfrac{1}{\varepsilon}+6\log\left(\dfrac{2e}{ \sin\theta_0}\right)$
\end{center}
or
\begin{equation}\label{d}
\ell(c)+10\log\dfrac{1}{\varepsilon}+8\log\dfrac{1}{\xi}+4d_{\gamma_0}+10\log \left(\frac{2e}{\sin\theta_0}\right)+\frac{\sin\theta_0}{e}+2\ell(\gamma_0)+4e^{-1}+2.
\end{equation}
Finally, after comparing upper bounds \ref{b}, \ref{c}, and \ref{d} in cases A, BA, and BB respectively, we set 
\begin{center}
$M(c,\varepsilon,\xi,X):=\ell(c)+10\log\dfrac{1}{\varepsilon}+8\log \dfrac{1}{\xi}+k'_X$
\end{center}
the upper bound of all cases, in which $k'_X:=4D+4d_{\gamma_0}+4+10\log \big(\frac{2e}{\sin\theta_0}\big)+\frac{\sin\theta_0}{e}+2\ell(\gamma_0)$ is a quantity that depends only on $X$.

\textbf{Part 3: Construction of the geodesic $\gamma$.} 

Since $X$ is orientable, $\gamma_0$ has two opposite sides denoted by $\gamma_0^{+}$ and $\gamma_0^{-}$. Let $\mu_{\pm}$ be an oriented geodesic arc from $\gamma_0$ to itself, orthogonal $\gamma_0$ in both endpoints, and which leaves and returns to $\gamma_0^{\pm}$. Note that these two geodesic arcs $\mu_{+}$ and $\mu_{-}$ are constructed using a finite cover of $X$ which lifts $\gamma_0$ to a simple closed geodesic. The lengths of these two arcs are constants depending on $X$.

In the previous parts, we replaced the collection  $\{c_i\}^N_{i=1}$ by a new collection, denoted by $\{\zeta_i\}^N_{i=1}$. We also defined a collection of the extended geodesic arcs of $\{\zeta_i\}^N_{i=1}$, denoted by $\{\zeta'_i\}^N_{i=1}$. In this collection, each element $\zeta'_i$, is of length at most $M(c,\varepsilon,\xi,X)$, has endpoints lying on $\gamma_0$ and forms good angles ($\geq \theta_0$) with $\gamma_0$, and is an extension of $\zeta_i$ by at least $r_{\varepsilon}$ in each direction. In short, each $\zeta_i$ is an example of the geodesic segment $c$ in Lemma \ref{epsilondistance}. Furthermore, we showed that any geodesic containing $\zeta_i$ in its $\varepsilon$-neighborhood contains $c_i$ in its $2\varepsilon$-neighborhood.

 With the new collection  $\{\zeta_i\}^N_{i=1}$ and its extension $\{\zeta'_i\}^N_{i=1}$ in hand, following exactly the same algorithm in the proof of \cite[Theorem 2.4]{BPS16}, one can construct a closed piecewise geodesic forming from these arcs with suitable choices of subsegments of the filling closed geodesic $\gamma_0$ and $\mu_{\pm}$ as following steps.
\begin{itemize}
\item Cyclically ordering and orienting each ${\zeta'_i}$ arbitrarily. 

\item If $\zeta'_{i+1}$ starts on the opposite side of $\gamma_0$ that $\zeta'_{i}$ ends on, we join the endpoint of $\zeta'_{i}$ to the starting point of $\zeta'_{i+1}$ by the shortest subarc of $\gamma_0$ which does this. 

\item If $\zeta'_{i+1}$ starts and $\zeta'_{i}$ ends on the same side, say $\gamma_0^{+}$, of $\gamma_0$, we join the endpoint of $\zeta'_{i}$ to the starting point of $\mu_{-}$ by the shortest subarc of $\gamma_0$ which does this. Then we join the endpoint of $\mu_{-}$ to the starting point of $\zeta'_{i+1}$ by the shortest subarc of $\gamma_0$ which does this.  

\end{itemize}
Note that each connecting shortest subarc introduced in each step is of length at most $\frac{\gamma_0}{2}$. The resulting closed piecewise geodesic, denoted by $\gamma'$, is contained in the $\varepsilon$-neighborhood of $\gamma$, where $\gamma$ is the unique closed geodesic in the free homotopy class of $\gamma'$. Due to the above construction, $\gamma$ is a nontrivial loop . Denote $\bar{c}$ the average length of the collection $\{c_i\}^N_{i=1}$, the length of $\gamma$ is bounded above by
 $$N(K_X+\bar{c}+10\log\frac{1}{\varepsilon}+8\log\frac{1}{\xi}),$$ where $K_X$ is a constant depending on $X$.
\end{proof}
A geodesic arc on $X$ is called a doubly truncated orthogeodesic on $X^\xi$ if it is perpendicular to the horocyclic boundary of $X^\xi$ at its endpoints. As a consequence of Theorem \ref{maintool}, we can also construct a doubly truncated orthogeodesic $\smallO$ with the same properties:\begin{thm}\label{maintool2}
For any $X\in \mathcal{M}_{g,n}$, there exists a constant $K_X$ such that the following holds. For all $0<\xi \leq 1$,   $0<\varepsilon \leq \min\{\log\frac{1}{\xi},1\}$, and any finite collection $\{c_i\}^N_{i=1}$ of geodesic arcs of average length $\bar{c}$ in $X^\xi$, there exist a doubly truncated orthogedesic $\smallO$ of length at most
\begin{center}
$(N+1)(K_X+\bar{c}+10\log\frac{1}{\varepsilon}+8\log\frac{1}{\xi})$
\end{center}
containing $\{c_i\}^N_{i=1}$ in its $2\varepsilon$-neighborhood.
\end{thm}
\begin{proof}
Let $P_0$ and $P_1$ be two arbitrary once-punctured polygons of the partition by $\gamma_0$ on $X$. Firstly, we will construct a doubly truncated orthogeodesic $\smallO_1$ with endpoints on the horocycles of length $1$ associated with the two polygons so that $\smallO_1$ contains $\{\zeta_i\}^N_{i=1}$ in its $\varepsilon$-neighborhood. We take the shortest one-sided orthogeodesic arc, denoted by $\zeta'_0$, oriented with the starting point on the horocycle of length $1$ of $P_0$ and the endpoint on $\gamma_0$. We take another shortest one-sided orthogeodesic arc, denoted by $\zeta'_{N+1}$, oriented with the starting point on $\gamma_0$ and the endpoint on the horocycle of length $1$ of $P_1$. For each $i\in \{1,2,...,N\}$, we orient $\zeta'_{i}$ arbitrarily. The new sequence $\{\zeta'_i\}^{N+1}_{i=0}$ is ordered linearly by its index.  We apply the connecting algorithm to this new sequence. Noting that $\zeta'_0$ is in the first step and $\zeta'_{N+1}$ is in the last step of the algorithm, one will obtain a doubly truncated orthogeodesic $\smallO_1$ as desired (see Figure \ref{ortho}).

Since $\smallO_1$ is an arc, it may not contain totally either $c_1$ or $c_N$ in its $2\varepsilon$-neighborhood. In this case, by applying Lemma \ref{epsilondistance} (P2), we only need to extend $\smallO_1$ by a small extra segment of length at most $\varepsilon$ in both directions. Note that, $\varepsilon \leq \log \frac{1}{\xi}$, and the distance between the horocycle of length $1$ and the horocycle of length $\xi$ is $\log \frac{1}{\xi}$, by extending $\smallO_1$ in both directions until it hits the boundary of $X^{\xi}$, we obtain $\smallO$ as desired.

\end{proof}
\begin{center}
\begin{minipage}{\linewidth}
\begin{center}
\includegraphics[width=0.4\linewidth]{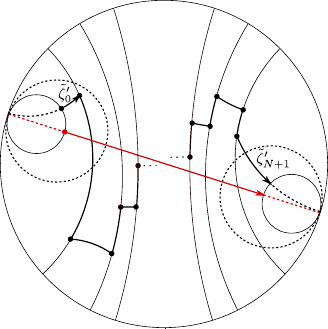}
\captionof{figure}{Lifting to $\mathbb{H}$.}
\label{ortho}
\end{center}
\end{minipage}
\end{center}


\section{Quantitative density on surface}
We now apply Theorem \ref{maintool} to prove results about quasi-dense geodesics.
\begin{thm}\label{delta2}
For all $X\in \mathcal{M}_{g,n}$ there exists a constant $C_X>0$ such that for all $0<\xi \leq 1$ and all $0<\varepsilon \leq 2$ there exists a closed geodesic $\gamma_\varepsilon$ that is $\varepsilon$-dense on $X^{\xi}$ and such that 
\begin{center}
$\ell(\gamma_\varepsilon) \leq C_{X} \dfrac{1}{\varepsilon}\bigg(\log \dfrac{1}{\varepsilon}+\log \dfrac{1}{\xi}\bigg)$.
\end{center}
\end{thm}
\begin{proof}
On $\mathbb{H}$, there is a fundamental polygon $F$ whose boundary consists of $4g+2n$ paired geodesic segments (or rays) which, when glued in pairs, turn the polygon into $X$. This polygon has $n$ ideal vertices and $4g+n$ ordinary vertices. See Figure \ref{g=1} for an example.
\begin{center}
\begin{minipage}{\linewidth}
\begin{center}
\includegraphics[width=0.4\linewidth]{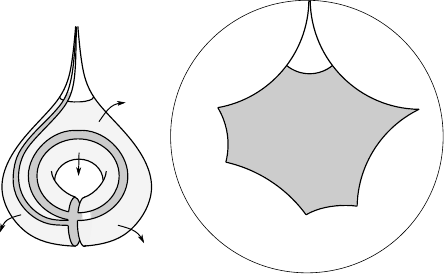}
\captionof{figure}{An example when g=1, n=1.}
\label{g=1}
\end{center}
\end{minipage}
\end{center}
 Since $X^{2}\subset X^{\xi} \subset  X$, there is a fundamental polygon of $X^{\xi}$ in $F$, say $F^{\xi}$, and a fundamental polygon of $X^2$ in $F^{\xi}$, say $F^2$. We note that the boundary of $F^2$ consists of $n$ horocyclic segments of length $2$ and $4g+2n$ geodesic segments. By replacing each horocyclic segment with a geodesic segment of length $2\arcsinh1$ with the same endpoints, we obtain the convex hull of $F^2$, denoted by $CH(F^2)$. We denote by $P_X$ the perimeter of $CH(F^2)$, and note that this value depends only on $X$.  On an edge of $CH(F^2)$, we choose the first point at a vertex, then choose the next points such that the segment on the boundary connecting two consecutive points is of length $\varepsilon$. If the length of the segment connecting the last point and the remaining vertex of the same edge is less than $\varepsilon$, that vertex will be chosen as the first point of the next edge, we then continue the choosing process. Eventually, we have chosen at most 
\begin{center}
$\dfrac{P_X}{\varepsilon}+4g+2n$
\end{center}
points on the boundary of $CH(F^2)$. See Figure \ref{cover}.
\begin{center}
\begin{minipage}{\linewidth}
\begin{center}
\includegraphics[width=0.4\linewidth]{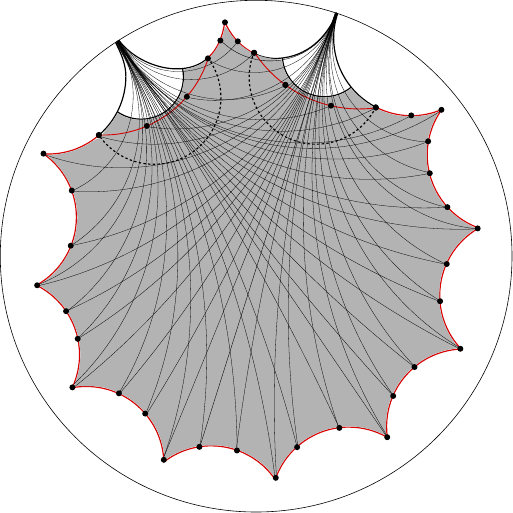}
\captionof{figure}{An example when g=2, n=2. Note that $CH(F^2)$ is the polygon with red edges.}
\label{cover}
\end{center}
\end{minipage}
\end{center}
Now we connect each ideal vertex to the points on the boundary of $CH(F^2)$. Since $CH(F^2)$ is convex, the parts of those geodesic rays in $F^{\xi}$ are exactly one-sided orthogeodesic segments. By gluing back paired geodesic segments of $F$ in pairs, these segments become one-sided orthogeodesic arcs on $X$ and we have at most
\begin{center}
$n\bigg(\dfrac{P_X}{\varepsilon}+4g+2n\bigg)$
\end{center}
one-sided orthogeodesic arcs on $X$. By construction, each segment is of length at most 
\begin{center}
$\dfrac{P_X}{2}+\log\bigg(\dfrac{2}{\xi}\bigg)$.
\end{center}
Moreover, the collection of the one-sided orthogeodesic arcs is $\frac{\varepsilon}{2}$-dense on $X^\xi$.
Thus by applying Theorem \ref{maintool} to this collection, we obtain the closed geodesic $\gamma_\varepsilon$ containing every arcs in its $\frac{\varepsilon}{2}$ neighborhood where length satisfies

\begin{equation}\label{x}
\ell(\gamma_\varepsilon) \leq n\bigg(\dfrac{P_X}{\varepsilon}+4g+2n\bigg)\bigg(K_X+\dfrac{P_X}{2}+\log\dfrac{2}{\xi}+10\log\dfrac{2}{\varepsilon}+8\log\dfrac{1}{\xi}\bigg).
\end{equation}
Then by manipulating the right hand of Inequality \ref{x}, we obtain a constant $C_X$ depending only on $X$ so that:
\begin{center}
$\ell(\gamma_\varepsilon) \leq C_X \dfrac{1}{\varepsilon}\bigg( \log\dfrac{1}{\varepsilon}+\log\dfrac{1}{\xi}\bigg).$
\end{center}
\end{proof}
By using the same collection of geodesic segments as in Theorem \ref{delta2}, and by the connecting algorithm in the proof of Theorem \ref{maintool2}, we also obtain the following result:
\begin{thm}\label{delta3}
Let $X\in \mathcal{M}_{g,n}$, there exists a constant $D_{X}>0$ such that for all $0<\xi \leq 1$ and all $0<\varepsilon \leq \min\{2\log\frac{1}{\xi},2\}$ there exists a doubly truncated orthogeodesic $\smallO_\varepsilon$ that is $\varepsilon$-dense on $X^\xi$ and such that 
\begin{center}
$\ell(\smallO_\varepsilon) \leq D_{X} \dfrac{1}{\varepsilon}\bigg(\log \dfrac{1}{\varepsilon}+\log \dfrac{1}{\xi}\bigg)$.
\end{center}
\end{thm}
We end this section with a corollary of Theorem \ref{delta2} where we apply Theorem 1.2 \cite{basmajian2013universal} to obtain an upper bound on the number of self-intersections of the quasi $\varepsilon$-dense closed geodesic. 
\begin{cor}\label{intersection}
Let $X\in \mathcal{M}_{g,n}$, there exists a constant $C_{X}>0$ such that for all $0<\xi \leq 1$ and all $0<\varepsilon \leq 2$ there exists a closed geodesic $\gamma_\varepsilon$ that is $\varepsilon$-dense on $X^\xi$ and such that 
\begin{center}
$2i(\gamma_\varepsilon,\gamma_\varepsilon) \leq C_{X}^{\frac{1}{\varepsilon}\big(\log \frac{1}{\varepsilon}+\log \frac{1}{\xi}\big)}$.
\end{center}
\end{cor}

\paragraph*{Acknowledgments.} This work was done by the author during his Ph.D. at the University of Luxembourg from 2018-2022 funded by the Luxembourg National Research Fund (FNR) PRIDE15/10949314/GSM. The author is very grateful to his thesis advisor Hugo Parlier for his thorough reading of the manuscript and many helpful conversations. The author also thanks Binbin Xu for useful discussions and the referee for several constructive comments that helped improve the article.

\bibliographystyle{amsplain}
\bibliography{ref}

  \small{
 {\it Address:}\\Department of Mathematics, University of Luxembourg, Esch-sur-Alzette, Luxembourg \\ \& Institute of Mathematics, Vietnam Academy of Science and Technology, Vietnam \\
{\it Email:}\\dnminh@math.ac.vn}

\end{document}